\documentclass{amsart}
\usepackage{pifont}
\usepackage{amsfonts}

\usepackage{amscd,amssymb,amsmath,graphicx,verbatim}
\usepackage[dvips]{hyperref}
\usepackage[TS1,OT1,T1]{fontenc}

\newtheorem{theorem}{Theorem}[section]
\newtheorem{lemma}[theorem]{Lemma}
\newtheorem{corollary}[theorem]{Corollary}
\newtheorem{proposition}[theorem]{Proposition}

\theoremstyle{definition}
\newtheorem{definition}[theorem]{Definition}
\newtheorem{example}[theorem]{Example}

\newtheorem{claim}[theorem]{Claim}
\newtheorem{question}[theorem]{Question}

\theoremstyle{remark}
\newtheorem{remark}[theorem]{Remark}

\numberwithin{equation}{section}

\begin{document}

\title{Schauder Bases and Operator Theory III: Schauder Spectrums}

\author{Yang Cao}
\address{Yang Cao, Department of Mathematics , Jilin university, 130012, Changchun, P.R.China} \email{Caoyang@jlu.edu.cn}

\author{Geng Tian}
\address{Geng Tian, Department of Mathematics , Jilin university, 130012, Changchun, P.R.China} \email{tiangeng09@mails.jlu.edu.cn}

\author{Bingzhe Hou}
\address{Bingzhe Hou, Department of Mathematics , Jilin university, 130012, Changchun, P.R.China} \email{houbz@jlu.edu.cn}

\date{Oct. 14, 2010}
\thanks{}
\begin{abstract}
In this paper, we study spectrums of Schauder operators.
We show that we always can choose a Schauder operator in a given orbit such that the Schauder spectrum of it
is empty.
\end{abstract}
\maketitle

\section{Introduction}

To study operators on $\mathcal{H}$ from a basis theory
viewpoint, it is naturel to consider the behavior of operators related by
equivalent bases. For examples, we show that there always be some
strongly irreducible operators in the orbit of equivalent Schauder matrices(\cite{Ji-Tian-Cao}).
However, in the usual way a spectral method consideration of operators in the equivalent orbit
is also important to the joint research both on operator theory and Schauder bases. For this reason, we introduces
the conception \textsl{Schauder spectrum} to do this work. The main
purpose of this paper is to show that we always can choose a Schauder operator in a given orbit such that the Schauder spectrum of it
is empty. The operator theory description of
bases on a separable Hilbert
space $\mathcal{H}$ developed in our paper \cite{Cao-1} helps us to do this job.

Recall that a sequence of vectors $\psi=\{f_{n}\}_{n=1}^{\infty}$ in $\mathcal{H}$ is said
to be a {\it Schauder basis} \cite{Singer} for $\mathcal{H}$ if
every element $f\in \mathcal{H}$ has a unique series expansion $f=\sum c_nf_n$
which converges in the norm of $\mathcal{H}$.
If
$\psi=\{f_{n}\}_{n=1}^{\infty}$ is Schauder basic for $\mathcal{H}$, the {\it sequence space
associated with} $\psi$ is defined to be the linear space of all
sequences $\{c_n\}$ for which $f=\sum c_nf_n$ is convergent. Two
Schauder bases $\{f_n\}_{n=1}^{\infty}$ and $\{g_n\}_{n=1}^{\infty}$ are {\it equivalent} to each other
if they have the same sequence space.
Denote by $\omega$ the countable infinite cardinal. In paper \cite{Cao-1}, we
considered the $\omega\times\omega$ matrix whose
column vectors comprise a Schauder basis and call them the {\it
Schauder matrix}. An operator has a Schauder matrix representation
under some ONB is called a {\it Schauder operator}. Given an orthonormal basis(ONB in short) $\varphi=\{e_n\}_{n=1}^{\infty}$,
the vector $f_{n}$ in a Schauder basis sequence $\psi=\{f_n\}_{n=1}^{\infty}$
corresponds an $l^{2}$ sequence $\{f_{mn}\}_{m=1}^{\infty}$ defined uniquely by the series
$f_{n}=\sum_{m=1}^{\infty} f_{mn}e_{m}$. The matrix $F_{\psi}=(f_{mn})_{\omega \times \omega}$ is called
the Schauder matrix of basis $\psi$ under the ONB $\varphi$.

Assume that $\psi_{1}, \psi_{2}$ are Schauder bases and $T_{\psi_{1}}, T_{\psi_{2}}$ are the
operators defined by Schauder matrices $F_{\psi_{1}}$ and $F_{\psi_{2}}$ respectively
under the same ONB. These operators $T_{\psi_{1}}, T_{\psi_{2}}$ will be called equivalent Schauder operators
if and only if $\psi_{1}, \psi_{2}$ are equivalent Schauder bases.
From the Arsove's theorem(\cite{Arsove}, or theorem 2.12 in \cite{Cao-1}), there is some invertible
operator $X\in L(\mathcal{H})$ such that $XT_{\psi_{1}}= T_{\psi_{2}}$ holds. Hence it is an equivalence relation on $L(\mathcal{H})$.
For a Schauder basis $\psi=\{f_{n}\}_{n=1}^{\infty}$, the set defined as
$$
\mathcal{O}_{gl}(\psi)=\{X\psi; X \in gl(\mathcal{H})\}
$$
in which $X\psi=\{Xf_{n}\}_{n=1}^{\infty}$ and $gl(\mathcal{H})$ consists of all invertible operators in $L(\mathcal{H})$
contains exactly all equivalent bases to $\psi$. Moreover, the set
$$
\mathcal{O}_{gl}(F_{\psi})=\{M_{X}F_{\psi}; M_{X} \hbox{ is the matrix of some operator }X \in gl(\mathcal{H})\}
$$
consists of all Schauder matrix equivalent to $F_{\psi}$. In the operator level, we define
$$
\mathcal{O}_{gl}(T_{\psi})=\{XT_{\psi}; X \in gl(\mathcal{H})\}.
$$
Then the set $\mathcal{O}_{gl}(T_{\psi})$ consists of operators related to bases equivalent to $\psi$.
Similarly, we consider following sets:
$$
\begin{array}{c}
\mathcal{O}_{u}(\psi)=\{U\psi; U \in U(\mathcal{H})\}, \\
\mathcal{O}_{u}(F_{\psi})=\{M_{U}F_{\psi}; M_{U} \hbox{ is the matrix of some unitary operator }U\}, \\
\mathcal{O}_{u}(T_{\psi})=\{UT_{\psi}; U \in U(\mathcal{H})\}.
\end{array}
$$
Roughly speaking, by these set we bind operators related to equivalent bases of the basis $\psi$ with the same
basis const. Since a Schauder operator $T_{\psi}$ is injective and having a
dense range in $\mathcal{H}$, if let $T_{\psi} = UA_{\psi}$ denote the polar decomposition
of $T_{\psi}$, then the partial isometry $U$  must be a unitary operator.
Hence, if $T_{\psi}$ is a Schauder operator and $T_{\psi} = UA_{\psi}$ denote the polar
decomposition  of $T$, then $\mathcal{O}_U(T_{\psi})= \mathcal{O}_U(A_{\psi})$,
where $A_{\psi}$ is an self-adjoint operator.

Now we state our main result in this paper.
\begin{theorem}\label{Theorem: Main Theorem of 3rd Paper}
For each Schauder operator $T$, there is an operator $T^{'} \in O_{u}(T)$ such that $\sigma_{S}(T^{'})=\emptyset$.
\end{theorem}

Above theorem there may be
notable differences between equivalent Schauder operators $T_{\psi_{1}}$ and $T_{\psi_{2}}$ from
the view of operator theory. For example, a self-adjoint $A$ may satisfy $\sigma_{S}(A)=\sigma(A)$ while there is some
unitary operator $U$ such that $\sigma_{S}(UA)=\emptyset$ holds. Moreover, we can choose a unitary operator $U$
as a unitary spread, which has a nice basis understanding.

We organize this paper as follows. In section 2, we give some examples and a description of Schauder spectrums
of compact operators.
In the case that the Schauder operator $T$ is a compact shift,
theorem \ref{Theorem: Main Theorem of 3rd Paper} is easy to check(see example \ref{ExES}).
The proof of the general situation is the content of section 3.

\section{Schauder Spectrum}\label{SchauderSpec}
In this subsection, we consider the spectrum of operators from the viewpoint of basis. Compare to the classical
results, there are many similar conclusions in the case of compact operators.

We begin with the following observation.

\begin{theorem}\label{Theorem: main theorem 2 sec 3}
The operator $T \in L(\mathcal{H})$ is a Schauder operator if and only if $T$ is injective and its range is dense in $\mathcal{H}$.
\end{theorem}

\begin{definition}
For a complex number $\lambda$, $\lambda$ will be called in the {\textsl Schauder spectrum} denoted by $\sigma_{S}(T)$ if
and only if there is no ONB such that $\lambda I-T$ has a matrix representation as a Schauder matrix. The set
$\rho_{S}(T)=\mathbb{C}-\sigma_{S}(T)$ will be called the \textsl{ Schauder resolvent set} of $T$.
\end{definition}

A direct result of theorem \ref{Theorem: main theorem 2 sec 3} is
\begin{theorem}{\label{Basic Pros}}
$\lambda \notin \sigma_{S}(T)$ if and only if $T$ is both injective and having a dense range in $\mathcal{H}$.
\end{theorem}

With above theorem, it is easy to check
\begin{proposition}
For a self-adjoint operator $A$, we have $\sigma_{S}(A)=\sigma(A)/\sigma_{p}(A)$.
\end{proposition}

\begin{example}
Assume that $[a, b]$ is an interval and $A$ be a self-adjoint operator satisfying $\sigma_{p}(A)=\emptyset, \sigma(A)=[a, b]$.
Then we have $\sigma_{S}(A)=\emptyset$.
\end{example}

\begin{example}
Consider the diagonal operator $D=diag(m_{1}, m_{2}, \cdots, m_{k}, \cdots)$ in which $m_{k} \ne 0$ and
$m_{k} \rightarrow 0$ for $k=0, 1, \cdots$. Then we have $\sigma_{S}(D)=\{m_{k}; k=1, 2, \cdots\}$. As an example,
diagonal operator $D=(1, \frac{1}{2}, \frac{1}{3},\cdots)$ has Schauder spectrum $\sigma_{S}(D)=\{\frac{1}{k}; k=1, 2, \cdots\}$.
\end{example}

More general, we have

\begin{proposition}
For any operator $T \in L(\mathcal{H}), X \in Gl(\mathcal{H})$, we have $0 \in \sigma_{S}(T)$ if and only if $0\in \sigma_{S}(XT)$.
\end{proposition}

\begin{corollary}
For self-adjoint operator $T$ we have $\sigma_{S}(T) \in \mathcal{R}$; For compact operator $K$ we have
$\sigma(K)/\{0\} \subseteq \sigma_{S}(K)$ and $0 \in \sigma_{S}(K)$ if and only if $0$ is in the set $\sigma_{p}(K)$
or $\overline{ran K}\ne \mathcal{H}$.
\end{corollary}
\begin{proof}
The assertion of first result of corollary is just the direct corollary of theorem \ref{Basic Pros}. If $K$ is a compact operator,
then its spectrum consist of $\{0\}$ and point spectrum $\sigma(K)$. In the case $0 \in \sigma_{p}(K)$
or $\overline{ran K}\ne \mathcal{H}$, $0$ is simply in the $\sigma_{S}(K)$;
if it is not true, then we have $K=UA$ in which $U$ is a unitary operator and $A$ is a compact                                             self-adjoint operator whose eigenvectors spans the Hilbert space $\mathcal{H}$.
\end{proof}

\begin{theorem}
For a compact self-adjoint operator $K$, $0 \in \sigma_{S}(K)$ if and only if $0 \in \sigma_{p}(K)$.
\end{theorem}

\begin{example}\label{ExES}
In the case that $K$ is a compact operator but its spectrum is equivalent to $\{0\}$, there is example in which $\sigma_{S}(K)=\emptyset$.
Consider an injective bilateral shift which also be a compact operator. For example, let $\{\tilde{e}_{j}\}_{j \in \mathbb{Z}}$ be an ONB
of $\mathcal{H}$ and $w_{j}=\frac{1}{1+|j|}$ for $j \in \mathbb{Z}$.
Then $K\tilde{e}_{j}=w_{j}\tilde{e}_{j-1}$ is such a compact injective bilateral
weighted shift operator(CIBWS, in short). As well-known that the spectrum of a weighted shift $T$ always be symmetric
(see \cite{Conway}, corollary 1 and 2, p52), that is, if
$\lambda \in \sigma(T)$ then we have $e^{i\theta}\lambda \in \sigma(T)$. So we must have $\sigma(K)=\{0\}$. Now to decide wether
$\sigma_{S}(K)=\{0\}$ or not, we need more information given by polar decomposition of $K$. To avoid complex computation on
$K^{*}K$, we need to rearrange the ONB $\{e_{j}\}_{j \in \mathbb{Z}}$ in some proper order as follows.
We take all $\tilde{e}_{j}$ with negative index $j<0$ as even integer and the one with index $j>0$ as positive integer.
More clarity, let
$$
e_{2k}=\tilde{e}_{-k}, e_{2k+1}=\tilde{e}_{k} \hbox{ for k } \in \mathbb{N}\cup\{0\}.
$$
Under the ONB $\{e_{k}\}_{k=1}^{\infty}$, the classical backward bilateral shift is just the operator
$$
Se_{2k+1}=e_{2k-1}, Se_{2k-2}=e_{2k} \hbox{ for }k=1, 2, \cdots.
$$
Now the CIBWS $K$ defined above can be rewritten as
$$
Ke_{2k+1}=w_{k}e_{2k-1}, Ke_{2k-2}=w_{-k}e_{2k}.
$$
Now consider the diagonal operator $D=diag(d_{j})$ with element $d_{j}$ on diagonal line in which $d_{2j-1}=w_{j}, d_{2j}=w_{-j}$,
then we have $K=SD$ which implies that $0$ is not in the Schauder spectrum $\sigma_{S}(K)$. Therefore we have $\sigma_{S}(K)=\emptyset$.
\end{example}

Compare to the classical Riesz's theorem on compact operator(see \cite{Conway}, theorem7.1, p219), we can characterize the Schauder
spectrum of compact operator as follows:

\begin{theorem}
If $\mathcal{H}$ is a separable Hilbert space and $\dim \mathcal{H}=\infty$. Then for a compact operator $K \in L(\mathcal{H})$,
one and only one of the following situations occurs: \\
1. $\sigma_{S}(K)=\emptyset$; \\
2. $\sigma_{S}(K)=\{0\}$;\\
3. $\sigma_{S}(K)=\{\lambda_{1}, \lambda_{2}, \cdots, \lambda_{n}\}$ in which $\lambda_{k} \ne 0$ and $\dim \ker(\lambda_{k} - K)<\infty$;\\
4. $\sigma_{S}(K)=\{0, \lambda_{1}, \lambda_{2}, \cdots, \lambda_{n}\}$ in which $\lambda_{k} \ne 0$ and $\dim \ker(\lambda_{k} - K)<\infty$;\\
5. $\sigma_{S}(K)=\{\lambda_{1}, \lambda_{2}, \cdots\}$. $0$ is the unique limit point of $\lambda_{k}$
and $\lambda_{k} \ne 0, \dim \ker(\lambda_{k} - K)<\infty$. \\
6. $\sigma_{S}(K)=\{0, \lambda_{1}, \lambda_{2}, \cdots\}$. $0$ is the unique limit point of $\lambda_{k}$
and $\lambda_{k} \ne 0, \dim \ker(\lambda_{k} - K)<\infty$.
\end{theorem}

\section{The Orbits of Schauder Operators and Schauder matrices}\label{OrbitofSchauderMatrix}


\subsection{}
Now we fix an ONB $\{e_{n}\}_{n=1}^{\infty}$. For a Schauder operator $T$, suppose $W$ be a unitary matrix(Hence a well defined
operator under the ONB fixed) such that $AW$ be a Schauder matrix. Now the set
$$
O_{gl}(T)=\{XT; X \in Gl(\mathcal{H})\}
$$
gives exactly the Schauder matrices(operators) $XTW$ whose corresponding basis is equivalent to the basis consisting of
the column vectors of $TW$. As well known that the topology group $Gl(\mathcal{H})$ is connected under the norm topology.
Hence roughly speaking two equivalent basis can always ``deform'' to each other. Moreover, if we ask that $X$ be a unitary operator,
then this ``deformation'' may have more nice properties. Following proposition is an example.

\begin{proposition}
Suppose that $F$ is a Schauder matrix and $U$ be a unitary operator. Then the basis given by $F$ and $UF$ are equivalent
basis with the same basis const. Moreover, if $F$ is a unconditional basis, then they have the same unconditional basis const.
\end{proposition}
\begin{proof}
For any projection $P$, we have $||UPU^{*}||=||P||$. Then apply proposition 2.6 and 2.7 in \cite{Cao-1}.
\end{proof}

Compare to the set $O_{gl}(T)$, we consider the set
$$
O_{u}(T)=\{UT; U \in U(\mathcal{H})\}.
$$
It just gives a part of equivalent basis of $A$ with the same basis const, although not all in general.
However, this situation is more interesting since it have a natural operator theory understanding, so called,
the polar decomposition of operator. In fact, since $T$ is injective and having a dense range in $\mathcal{H}$
(proposition 2.14, \cite{Cao-1}),
we know that the partial isometry $U$ appearing in its polar decomposition $T=UA$ must be a unitary operator.
Hence we have $O_{u}(T)=O_{u}(A)$ if $T$ is a Schauder operator. This fact suggest us that to study the Schauder operator
we can begin with the self-adjoint operators having a dense range and then consider their orbit $O_{u}(A)$
(cf, papers \cite{Y2}, \cite{Ole}).

A natural question is
\begin{question}
Assume that $F_{1}, F_{2}$ are equivalent Schauder bases and $T_{1}, T_{2}$ be the corresponding operators.
Does there be some notable difference between the operators $T_{1}$ and $T_{2}$?
\end{question}

From the operator theory viewpoint, the answer is affirmative. We shall show that even in the case $T_{1}=UT_{2}$,
their spectrum may be very different. In fact, we have

\begin{theorem}\label{Maintheorem3}
Assume that $A$ is a self-adjoint operator such that $0 \notin \sigma_{p}(A)$. Then there is some unitary operator $U$
such that $\sigma_{p}(UA)=\sigma_{p}(AU^{*})=\emptyset$. Moreover, we can choose the unitary operator $U$ as a combination
of unitary spreads.
\end{theorem}

By virtue of \ref{Maintheorem3}, we can always choose a good representative element
from the set $O_{u}(T)$. That is the following theorem.

\begin{theorem}\label{Theorem: Main Theorem of 3rd Paper}
For each Schauder operator $T$, there is an operator $T^{'} \in O_{u}(T)$ such that $\sigma_{S}(T^{'})=\emptyset$.
\end{theorem}
\begin{proof}
By virtue of theorem 2.13 in \cite{Cao-1}, we need only to verify the following claim:
\begin{claim}
For a self-adjoint operator $A$, there be some unitary operator $U$ such that
the operator $\lambda I-UA$ always be injective and has a dense range
in $\mathcal{H}$ for each $\lambda \in \mathbb{C}$.
\end{claim}
Now by virtue of theorem \ref{Maintheorem3}, there is some unitary operator $U$ such that
$\sigma_{p}(UA)=\sigma_{p}(AU^{*})=\emptyset$. From $\sigma_{p}(UA)=\emptyset$, we know that the operator
$\lambda I-A$ always be injective; On the other side, basic operator theory result tell us
$$
\overline{Ran(\lambda I-UA)}=(\ker(\lambda I-UA)^{*})^{\perp}.
$$
So from $\sigma_{p}(AU^{*})=\emptyset$ we have
$$
\overline{Ran(\lambda I-UA)}=(\ker(\bar{\lambda} I-AU^{*}))^{\perp}=\mathcal{H}.
$$
\end{proof}

We shall prove theorem \ref{Maintheorem3} in later subsections. Prior to this, we give some remarks at first.

\begin{remark}
Relation to the ``invariant subspace'' problem. \\
1. It is trivial to check that we have $\sigma_{S}(T)=\emptyset$ if $T$ has no nontrivial subspace. \\
2. Assume that there do have some operator $T$ having no nontrivial
invariant subspace. Then $T$ appear in some orbit $O_{u}(A)$ of some self-adjoint operator $A$ since $T$ must be
a Schauder operator(injective and having a dense range). What can we say about the self-adjoint operator $A$? Clearly
there do exist some orbit $O_{u}(A)$ such that each operator in it must have a nontrivial invariant subspace. A trivial
example is the identity operator $I$(cf, \cite{Beurling}, or IX.9 \cite{Conway}). \\
3. Theorem \ref{Maintheorem3} tell us that we can remove the eigen-subspaces, that is, the most ``trivial'' nontrivial subspaces.
\end{remark}

\begin{remark}Continuous ``deformation'' of Schauder bases.
If we restrict to consider the basis whose corresponding Schauder matrix represents a bounded operator, then we can
define the \textsl{ continuous deformation} of bases as follows. A (continuous) curve of bases is just a map
$$
\gamma : I \rightarrow L(\mathcal{H})
$$
satisfying the following properties:\\
1. for each $t \in I$, $\gamma(t)$ is a Schauder matrix; \\
2. $\gamma(t)$ represents a bounded operator; \\
3. The map is continuous in the variable $t \in I$ under the norm topology $L(\mathcal{H})$. \\
Here $I$ is an interval(either open or closed).
Denote by $\mathcal{F}$ the set of all Schauder matrices, we have the following question:
\begin{question}
Does $\mathcal{F}$ must be a connected set?
\end{question}
Given a Schauder matrix $F$, denote by $O_{gl}(F)$
the set consisting of all Schauder matrices equivalent to $F$.
As well-known, invertible operators are connected(Problem 141, \cite{Halmos}, p76), so we have
\begin{theorem}
The set $O_{gl}(F)$ is always path-connected for each Schauder matrix $F$.
\end{theorem}

Denote by $O_{gl}^{c}(F)$ the set of Schauder matrices
$F$ which is a Schauder matrix and there is a sequence $F_{k} \in O_{gl}(F)$ such that $||F_{k}-F || \rightarrow 0$.
\begin{question}
If $F$ is a conditional(unconditional) matrix, whether each matrix $F^{'}\in O^{c}_{gl}(F)$ must be
also a conditional(unconditional)
matrix or not?
\end{question}
\end{remark}


\subsection{}Before going ahead, recall the definition of the ``\textsl{spread from $A$ to $B$}'' given by W. T. Gowers and
B. Maurey.

\begin{definition}(\cite{Gowers2}, p549)
Given an ONB $\{e_{n}\}_{n=1}^{\infty}$
and two infinite subsets $A, B$ of $\mathbb{N}$. Let $c_{00}$ be the vector
space of all sequences of finite support. Let the elements of $A$ and $B$ be written in increasing order respectively
as $\{a_{1}, a_{2}, \cdots\}$ and $\{b_{1}, b_{2}, \cdots\}$. Then $e_{n}$ maps to $0$ if $n \notin A$, and $e_{a_{k}}$
maps to $e_{b_{k}}$ for every $k \in \mathbb{N}$. Denote this map by $S_{A, B}$ and call it the spread from $A$ to $B$.
\end{definition}

\begin{example}(\cite{Gowers2}, p549)
Let $A=\{2, 3, 4, \cdots\}$ and $B=\{1, 2, 3\}$, then $S_{A, B}$ is just the backward unilateral shift operator(cf, \cite{Shields})
which is defined as $S(e_{n})=e_{n-1}$ for $n \ge 2$ and $Se_{1}=0$.
\end{example}

\begin{example}\label{Example: Unitary Op is LC of Spreads}
Using spread forms, we can write some unitary operator into their
linear combination. For example, let $\sigma$ be a bijection on $\mathbb{N}$(a permutation of $\mathbb{N}$,
so called in \cite{Singer}) defined as $\sigma(2n)=2(n-1)$ for $n \ge 2$ and $\sigma(2)=1$ for even numbers and
$\sigma(2n-1)=2n+1$ for odd numbers. Then the operator $U_{\sigma}(e_{n})=e_{\sigma(n)}$ is a bilateral shift and a unitary operator.
Let $A_{2}=\{2, 4, 6, \cdots\}, B_{2}=\{1, 2, 4, \cdots\}$ and $A_{1}=\{1, 3, 5, \cdots\}$ and $A_{2}=\{3, 5, 7, \cdots\}$.
We have $U_{\sigma}=S_{A_{1}, B_{1}}+S_{A_{2}, B_{2}}$.
\end{example}

\begin{definition}
A unitary operator $U$ on $\mathcal{H}$ is said to be a unitary spread if there is a sequence $\{S_{A_{n}, B_{n}}\}_{n=1}^{\infty}$ of spreads
such that the series $\sum_{n=1}^{\infty} S_{A_{n}, B_{n}}$ converges to $U$ in SOT. Moreover, $U$ will be called a finite unitary spread
if $U$ can be written as a finite linear combination.
\end{definition}

In the paper \cite{Cao}, we proved the following result.
\begin{lemma}\label{Lemma: Bijection on N is a unitary Spread}
For each bijection $\sigma$ on the set $\mathbb{N}$, the unitary operator $U_{\sigma}$ is a unitary spread.
\end{lemma}


\subsection{}Now we begin to prove theorem \ref{Maintheorem3}. Firstly we give an outline of the proof.
By the spectral theorem of normal operators, we write a self-adjoint operator into the orthogonal diagonal
direct sum $A=A_{0} \oplus A_{1}$ in which these operators satisfy the following properties.

Property 1. The eigenvectors of $A_{0}$ defined on the Hilbert space $\mathcal{H}_{1}$
span the whole Hilbert space $\mathcal{H}_{1}$;

Property 2. The operator $A_{1}$ defined on the Hilbert space $\mathcal{H}_{1}$ has only a ``small''
point spectrum. The meaning of ``small'' shall be clear in later proof.

Roughly speaking, $A_{0}$ represents the discrete case and $A_{1}$ the continuous one. Moreover,
in each situation, the spectrum containing the point $0$ or not will be considered by different ways.
We shall deal with the discrete case in this subsection and then turn to the continuous one later.

\begin{lemma}\label{BasicLemma}
Assume that $A$ is a self-adjoint operator satisfying the following properties:

1. $\sigma(A)=\sigma_{p}(A) \cup \{0\}$ and $0$ is the unique accumulation point of $\sigma(A)$;

2. For each $t \in \sigma_{p}(A)$, $\dim \ker{A-tI}=1$;

3. $R(A)$ is dense in the Hilbert space $\mathcal{H}$.

Then there is a unitary spread $U$ such that we have both $\sigma_{p}(UA)=\emptyset$ and $\sigma_{p}(AU^{*})=\emptyset$.
\end{lemma}

\begin{proof}
The self-adjoint operator satisfying the conditions appearing in the lemma has a spectrum in the following form:
$$
\sigma(A)=\{t_{1}, t_{2}, \cdots, t_{k}, \cdots\} \cup \{0\}
$$
where $t_{k} > t_{k+1}$ and $\{0\}$ is the only one accumulation point of the sequence $\{t_{k}\}$. It is clear that $A$ is a compact operator.
Moreover, each $t_{k}$ is a point spectrum of $A$ since $A$ is a self-adjoint operator and $t_{k}$ is a isolated point in $\sigma(A)$.
Then $A$ has a diagonal form as follows in an ONB:
$$
\begin{array}{cc}
\left(
 \begin{array}{ccccc}
  t_{1} &   0  &\cdots &   0  & \cdots \\
    0   & t_{2}&\cdots &   0  & \cdots \\
  \vdots&\vdots&\ddots &\vdots&         \\
     0  &  0   &\cdots & t_{k}&        \\
  \vdots&\vdots&       &      & \ddots
 \end{array}
\right) & \begin{array}{c}
           \mathcal{H}_{0} \\
           \mathcal{H}_{1} \\
           \vdots          \\
           \mathcal{H}_{k} \\
           \vdots
          \end{array}

\end{array}
$$
where by $\mathcal{H}_{k}$ we denote the 1-dimensional subspace $\ker(t_{k}I-A)$.

Now we begin to construct the unitary spread $U$. For convenience,
we denote by $e_{k}$ the unit eigenvector of $\ker(t_{k}I-A)$. Let $U$ be the shift constructed
as follows:
$$
\begin{array}{rl}
Ue_{n}= & \left\{
          \begin{array}{ll}
            e_{1}, &\hbox{ for }n=2; \\
            e_{2k-2}, &\hbox{ for }n=2k, k \ge 2 \\
            e_{2k+1}, &\hbox{ for }n=2k-1, k \ge 1
          \end{array}
          \right.
\end{array}.
$$
Clearly, the operator $U$ is just the unitary spread $U_{\sigma}$ defined in example \ref{Example: Unitary Op is LC of Spreads}.
We have
$$
\begin{array}{rl}
UAe_{n}= & \left\{
          \begin{array}{ll}
            t_{2}e_{1}, &\hbox{ for }n=2; \\
            t_{2k}e_{2k-2}, &\hbox{ for }n=2k, k \ge 2 \\
            t_{2k-1}e_{2k+1}, &\hbox{ for }n=2k-1, k \ge 1
          \end{array}
          \right.
\end{array}.
$$
Assume that $x=\sum_{k=0}^{\infty} x_{k}e_{k}$, then we have
$UAx=y=\sum_{k=0}^{\infty} y_{k}e_{k}$ where
$$
\begin{array}{rl}
y_{n}= & \left\{
          \begin{array}{ll}
            t_{2}x_{2}, &\hbox{ for }n=1; \\
            t_{2k+2}x_{2k+2}, &\hbox{ for }n=2k, k \ge 1 \\
            t_{2k-1}x_{2k-1}, &\hbox{ for }n=2k+1, k \ge 1
          \end{array}
          \right.
\end{array}.
$$
Now if $\lambda$ is an eigenvalue of $UA-\lambda I$, then we have
$$
\begin{array}{rcl}
t_{2}x_{2}      &=&\lambda x_{1} \\
t_{2k+2}x_{2k+2}&=&\lambda x_{2k} \\
t_{2k-1}x_{2k-1}&=&\lambda x_{2k+1}
\end{array}.
$$
Therefore we have
$$
\begin{array}{rl}
x_{2k+1}&=x_{1}\cdot \lambda^{-k} \cdot \prod_{j=1}^{k} t_{2j-1}, \\
x_{2k}&=x_{1}\lambda^{k} \cdot \prod_{j=1}^{k} \frac{1}{t_{2j}}.
\end{array}
$$
Now $\lambda^{k} \cdot \prod_{j=1}^{k} \frac{1}{t_{2j}} \rightarrow \infty$ for $\lambda \ne 0$ as $k \rightarrow \infty$
since $t_{j}$ tends to 0, we must have $x_{n}=0$ for $n=1, 2, \cdots$. Therefore we must have $\sigma_{p}(\lambda I-UA)=\emptyset$
for $\lambda \ne 0$. Moreover, by $\overline{Ran(A)}=\mathcal{H}$ we have $\ker(A)=\overline{Ran(A)}^{\perp}=\{0\}$.
Hence we have $\ker(UA)=\ker(A)=\{0\}$ and then $\sigma_{p}(UA)=\emptyset$ in turn.
$UA$ is just a weighted bilateral shift operator(cf, \cite{Shields}). Moreover, $UA$ is a compact operator since $A$ is compact itself.
By Riesz's theorem on compact operator, we have $\sigma(UA)=\{0\}$. So we just need to show
$\overline{Ran(UA)}=\mathcal{H}$ to finish the proof. But it is trivial by the fact $Ran(UA)=Ran(A)$.
\end{proof}

\begin{corollary}\label{Corollary: of Basic Lemma}
Assume that $A$ is a self-adjoint operator satisfying the following properties:

1. $\sigma(A)=\sigma_{p}(A) \cup \{0\}$ and $0$ is the unique accumulation point of $\sigma(A)$;

2. For each $t \in \sigma_{p}(A)$, $\dim \ker{A-tI}=1$;

3. $R(A)$ is dense in the Hilbert space $\mathcal{H}$.

Then there is a unitary operator $U$ such that $UA$ has no point spectrum. Moreover, we can ask that
the unitary operator $U$ satisfies the following property:\\
\indent for any point $\lambda \in \mathbb{C}$, $\lambda I-UA$ have a dense range in $\mathcal{H}$.
\end{corollary}

Now we get rid of the second condition of above lemma.

\begin{lemma}\label{Lemma: ONLYACC}
Assume that $A$ is a self-adjoint operator satisfying the following properties:

1. $\sigma(A)=\sigma_{p}(A) \cup \{0\}$ and $0$ is the unique accumulation point of $\sigma(A)$;

2. $R(A)$ is dense in the Hilbert space $\mathcal{H}$.

Then there is a unitary spread $U$ such that both $UA$ and $AU^{*}$ have empty point spectrum.
\end{lemma}

\begin{proof}
Assume $\sigma_{p}(A)=\{t_{k}\}_{k=1}^{\infty}$ and $t_{k} >t_{k+1}$ for $k \in \mathbb{N}$.
Firstly we cut the integers set $\mathbb{N}$ into two disjoint subsets:
$$
\begin{array}{l}
E_{0}=\{k \in \mathbb{N}; \dim \ker{A-t_{k}I}< \infty\}, \\
E_{1}=\{k \in \mathbb{N}; \dim \ker{A-t_{k}I}= \infty\}.
\end{array}
$$
Now define
$$
\begin{array}{l}
\mathcal{H}_{0}= span \{\ker(t_{k}I-A); k \in E_{0}\};\\
\mathcal{H}_{1}= span \{\ker(t_{k}I-A); k \in E_{1}\};
\end{array}
$$

Moreover, denote by $\widetilde{\mathcal{H}}_{k}=\ker(t_{k}I-A)$ and $I_{k}$ be the identity operator on
$\widetilde{\mathcal{H}}_{k}$. Let $A_{1}=\oplus_{k \in E_{0}}\tilde{A}_{k}$ and $A_{1}=\oplus_{k\in E_{1}}\tilde{A}_{k}$
in which we define $\tilde{A}_{k}=t_{k}I_{k}$.
We can write $A$ into the orthogonal direct sum $A=A_{0}\oplus A_{1}$.

From $E_{0}$ we construct a new set $E^{'}_{0}$ as follows. If $\dim \ker{A-t_{k}I}=k_{m}$, then we add
$k_{m}-1$ copies of $t_{k}$ into $E^{'}_{0}$. Then for each $t^{'}_{k}$, we can assign exactly one unit
vector $e_{k}^{(0)} \in \ker(t^{'}_{k}I-A)$ such that those vectors $\{e_{k}^{(0)}\}$ consists of an orthonormal subset.
Arrange the elements in $E^{'}_{0}$ decreasingly
as $t^{'}_{0}\ge t^{'}_{1}\ge t^{'}_{2}\ge \cdots \ge t^{'}_{k}\ge t^{'}_{k+1} \ge \cdots$.
If $E_{0}$ is a finite subset, then we must have $\lim_{k \in E_{1}, k\rightarrow \infty} t_{k}=0$. Let $N=\max E_{0}$.
For each $k>N$, fix a unit vector $\tilde{e}_{k}^{(0)} \in \widetilde{\mathcal{H}}_{k}$ and add $t^{'}_{k}=t_{k}$ into the set $E^{'}_{0}$.
By replacing $\mathcal{H}_{0}$ by the subspace $\mathcal{H}_{0}^{'}=span_{k>N}\{\tilde{e}_{k}^{0}\}\oplus\mathcal{H}_{0}$,
and $\mathcal{H}_{1}$ by $\mathcal{H}_{1}^{'}=(\mathcal{H}_{0}^{'})^{\perp}$, we always can assume that $E_{0}$ be a infinite subset and
then we have $\lim_{k \in E_{0}, k\rightarrow \infty} t_{k}=0$.
Moreover, let $\widetilde{\mathcal{H}}^{'}_{k}=(\tilde{e}^{(0)}_{k})^{\perp} \cap \widetilde{\mathcal{H}}_{k}$ for $k>N$
and $\widetilde{\mathcal{H}}^{'}_{k}=\widetilde{\mathcal{H}}_{k}$ for $k \le N$,
then clearly we have $\widetilde{\mathcal{H}}_{l}\perp \widetilde{\mathcal{H}}_{m}$ for $l \ne m$ and
$\mathcal{H}_{1}^{'}=\oplus_{k \in E_{1}} \widetilde{\mathcal{H}}^{'}_{k}$.
Now the operator $A$ has the following form
$$
\begin{array}{rll}
A=&\left(
   \begin{array}{cc}
    A^{'}_{0}&0 \\
      0  &A^{'}_{1}
   \end{array}
   \right)
& \begin{array}{c}
              \mathcal{H}^{'}_{0} \\
              \mathcal{H}^{'}_{1}
             \end{array}
\end{array}.
$$
The operator $A_{0}^{'}$ satisfies all conditions in lemma \ref{BasicLemma},
by modifying the unitary operator constructed in the proof of \ref{BasicLemma} by these new indices we can get a
unitary spread $U_{0}$ on the subspace $\mathcal{H}^{'}_{0}$ such that both $\sigma_{p}(U_{0}A^{'}_{0})=\emptyset$
and $\sigma_{p}(A^{'}_{0}U_{0}^{*})=\emptyset$ hold. Moreover, it is easy to check that there is some bijection
$\sigma_{0}$ on $E_{0}^{'}$ such that $U_{0}=U_{\sigma_{0}}$.

Now we consider the operator $A^{'}_{1}$. It can be written as the orthogonal direct sum $A=\oplus_{k \in E_{1}}\tilde{A}_{k}$
in which $\tilde{A}_{k}$ is just the restriction of $A$ on the infinite dimensional subspace $\widetilde{\mathcal{H}}^{'}_{k}$.
For each operator $k \in E_{1}$, choose an ONB $\{e^{(k)}_{l}\}_{l=1}^{\infty}$ of the subspace $\widetilde{\mathcal{H}}^{'}_{k}$.
Denote by $A^{(k)}_{1}=A_{1}, B^{(k)}_{1}=B_{1}, A^{(k)}_{2}=A_{2}$ and $B^{(k)}_{2}=B_{2}$ corresponding to the subsets of $\mathbb{N}$
defined in example \ref{Example: Unitary Op is LC of Spreads} and
$\widetilde{U}_{k}=S_{A^{(k)}_{1}, B_{1}^{(k)}}+S_{A^{(k)}_{2}, B^{(k)}_{2}}$.
Then we have $\widetilde{U}_{k}A_{k}=t_{k}\widetilde{U}_{k}$
which satisfies $\sigma_{p}(t_{k}\widetilde{U}_{k})=\sigma_{p}(t_{k}\widetilde{U}_{k}^{*})=\emptyset$.
Clearly the operator defined as $U_{1}=\oplus_{k=1}^{\infty} \widetilde{U}_{k}$ is a unitary operator on $\mathcal{H}_{1}$
and also satisfying $\sigma_{p}(U_{1}A_{1})=\sigma_{p}(A_{1}U_{1}^{*})=\emptyset$.
Example \ref{Example: Unitary Op is LC of Spreads} also tell us that there is some bijection $\sigma_{k}$
on $\mathbb{N}$ such that $\widetilde{U}_{k}=U_{\sigma_{k}}$ for each $k \in E_{1}$.

Now we turn to verify that the unitary operator
$$
\begin{array}{rll}
U=&
\left(
\begin{array}{cc}
U_{0} & 0\\
   0  & U_{1}
\end{array}
\right) & \begin{array}{c}
              \mathcal{H}^{'}_{0} \\
              \mathcal{H}^{'}_{1}
             \end{array}
\end{array}
$$
is the unitary spread we seek for. Clearly we only need to show that $U$ is a unitary spread.
Let $\mathbb{N}^{'}=E^{'}_{0}\times (\times_{k \in E_{1}} \mathbb{N})$, then clearly the set $\mathbb{N}^{'}$
is just $\mathbb{N}$ in a new order. The map defined as $\sigma: \mathbb{N} \rightarrow \mathbb{N}$
by $\sigma=\sigma_{0} \times (\times_{k \in E_{1}} \sigma_{k})$ is trivially a bijection. Therefore we can apply
lemma \ref{Lemma: Bijection on N is a unitary Spread} to finish the proof.
\end{proof}

Now we turn to the more general situation.

\begin{lemma}\label{Lemma: spectrum is just a finite point spectrum}
Assume that $A$ is a self-adjoint operator satisfying the following properties:

1. $\sigma(A)=\sigma_{p}(A)$ is a finite set;

2. $R(A)$ is dense in the Hilbert space $\mathcal{H}$.

Then there is a unitary spread $U$ such that both $UA$ and $AU^{*}$ have empty point spectrum.
\end{lemma}
\begin{proof}
Now we have $\sigma(A)=\sigma_{p}(A)=\{t_{1}, t_{2}, \cdots, t_{n}\}$ and $t_{k} \ne 0$. Let
$$
\begin{array}{l}
E_{0}=\{1\le k \le n; \dim \ker{A-t_{k}I}< \infty\}, \\
E_{1}=\{1\le k \le n; \dim \ker{A-t_{k}I}= \infty\}.
\end{array}
$$
Clearly we have $E_{1} \ne \emptyset$ since $\dim \mathcal{H}=\infty$. Without loss of generality,
assume $t_{1} \in E_{1}$. The subspace $\widetilde{\mathcal{H}}_{0}=span_{k \in E_{0}} \ker(t_{k}I-A)$
is a finite dimensional subspace, so we can pick an ONB $\{\tilde{e}^{(0)}_{l}\}_{l=1}^{N}$ of it in which
$N=\sum_{k \in E_{0}} \dim\ker(t_{k}I-A)$. Choose an ONB $\{e_{m}\}_{m=1}^{\infty}$ of the subspace $\ker(t_{1}I-A)$.
Let $e^{(0)}_{l}=\tilde{e}^{(0)}_{l}$ for $1\le l \le N$ and $e^{(0)}_{l}=e_{l-N}$ for $l \ge n$, then the sequence
$\{e^{(0)}_{l}\}_{l=1}^{\infty}$ is an ONB of the subspace $\mathcal{H}_{0}=\widetilde{\mathcal{H}}_{0}\oplus \ker(t_{1}I-A)$.
We also have $\mathcal{H}_{1}=\mathcal{H}_{0}^{\perp}=\oplus_{k \in E_{1}, k \ne 1}\ker(t_{k}I-A)$.
Now we can rewrite the operator $A$ into the following form
$$
\begin{array}{rll}
A=&\left(
   \begin{array}{cc}
    A_{0}&0 \\
      0  &A_{1}
   \end{array}
   \right)
& \begin{array}{c}
              \mathcal{H}_{0} \\
              \mathcal{H}_{1}
             \end{array}
\end{array}.
$$
Repeat the corresponding discussion in the proof of lemma \ref{Lemma: ONLYACC}, we need only to prove the following
\begin{claim}
There is a unitary spread $U_{0}$ on $\mathcal{H}_{0}$ such that we have both $\sigma_{p}(U_{0}A_{0})=\emptyset$
and $\sigma_{p}(A_{0}U^{*}_{0})=\emptyset$.
\end{claim}
To do this, let $U_{0}$ be the unitary spread $U_{\sigma}$ defined in example \ref{Example: Unitary Op is LC of Spreads}.
Now it is trivial to check that operators $U_{0}A_{0}$ and $t_{1}U_{0}$ are similarity to each other
(that is, there is some invertible operator $X \in L(\mathcal{H}_{0})$ such that we have $XU_{0}A_{0}X^{-1}=t_{1}U_{0}$)
by theorem 2 of the paper \cite{Shields}(p54). And then claim holds by the fact $\sigma_{p}(U_{0})=\sigma_{p}(U^{*}_{0})=\emptyset$.
\end{proof}

\begin{corollary}
If $T$ is a compact Schauder operator, then there is a unitary operator $U$ such that $\sigma_{S}(UT)=\{0\}$.
\end{corollary}

\begin{lemma}\label{Lemma: Point spectrum do have an Acc point}
Assume that $A$ is a self-adjoint operator, satisfying the following properties:

1. $span \{\ker(t_{k}I-A); t_{k} \in \sigma_{p}(A)\}=\mathcal{H}$;

2. There is some point $t_{0}$ such that it is an accumulation point of $\sigma(A)$.

3. $R(A)$ is dense in the Hilbert space $\mathcal{H}$.

Then there is a unitary operator $U$ such that both $UA$ and $AU^{*}$ have an empty point spectrum.
\end{lemma}
\begin{proof}
Firstly we cut the integers set $\mathbb{N}$ into two disjoint subsets:
$$
\begin{array}{l}
E_{0}=\{k \in \mathbb{N}; \dim \ker{A-t_{k}I}< \infty\}, \\
E_{1}=\{k \in \mathbb{N}; \dim \ker{A-t_{k}I}= \infty\}.
\end{array}
$$
And define
$$
\begin{array}{l}
\mathcal{H}_{0}= span \{\ker(t_{k}I-A); k \in E_{0}\};\\
\mathcal{H}_{1}= span \{\ker(t_{k}I-A); k \in E_{1}\}.
\end{array}
$$
Then we can write $A$ into the form
$$
\begin{array}{rll}
A=&\left(
   \begin{array}{cc}
    A_{0}&0 \\
      0  &A_{1}
   \end{array}
   \right)
& \begin{array}{c}
              \mathcal{H}_{0} \\
              \mathcal{H}_{1}
             \end{array}
\end{array}.
$$
Now with the same discussion on the part $A_{1}^{'}$ in lemma \ref{Lemma: ONLYACC}, we can remove $A_{1}$
since $t_{k} \ne 0$ by property 3. That is, we can assume $E_{0}=\mathbb{N}$ and $E_{1}=\emptyset$.

Let $\{a_{k}\}_{k=1}^{\infty}$ be a sequence of positive numbers satisfying $a_{k}\rightarrow 0$ decreasingly and
$\sum_{k=1}^{\infty} a_{k} <\infty$.
Assume  $\sigma(A)\subseteq [t_{0}-M, t_{0}+M]$ and $M>a_{1}$. Denote by
$$
\begin{array}{l}
I_{1}=[t_{0}-M, t_{0}-a_{1})\cup (t_{0}+a_{1}, t_{0}+M] \hbox{ and } \\
I_{k}=[t_{0}-a_{k}, t_{0}-a_{k+1}) \cup (t_{0}+a_{k+1}, t_{0}+a_{k}]
\hbox{ for }k \ge 2.
\end{array}
$$
Then $\sigma_{p}(A) \cap I_{k}$ contains at most countable elements. We divide $\mathbb{N}$ into two parts
$$
\begin{array}{l}
G_{0}=\{k \in \mathbb{N}; Card\{\sigma_{p}(A) \cap I_{k}\}< \infty\} \hbox{ and } \\
G_{1}=\{k \in \mathbb{N}; Card\{\sigma_{p}(A) \cap I_{k}\}= \infty\}.
\end{array}
$$
According to the cardinal of the set $G_{1}$, we shall prove the lemma in following two cases.

Case 1. If $G_{1}$ is an infinite subset, we can absorb the elements in $G_{0}$ by the first next one in the set $G_{1}$ and then
we can assume $G_{1}=\mathbb{N}$ and $G_{0}=\emptyset$. Moreover, we can also ask that $\dim \ker(t_{k}I=1)$ by adding
at most countable copies to $t_{k}$ and rearranging this new countable set.
For each $k \in \mathbb{N}$, we arrange the elements in the set
$\sigma_{p}(A) \cap I_{k}$ as a sequence $\{t_{l}^{(k)}\}_{l=1}^{\infty}$. Clearly we have
$\lim_{k \rightarrow \infty} t_{l}^{(k)}=t_{0}$. For each $t_{l}^{(k)}$, we assign a unit vector $e^{(k)}_{l}$.
Then by the spectral theorem we have $e_{l}^{(k)}\perp e_{m}^{(j)}$ for $(l, k)\ne (m, j)$.
Denote by $\widetilde{\mathcal{H}}_{l}=span_{k\in \mathbb{N}} \{e_{l}^{(k)}\}$. Now we can write $A$ into the orthogonal
direct sum $A=\oplus_{l}\tilde{A}_{l}$ in which the operator $\tilde{A}$ is the restriction of $A$ on the subspace
$\widetilde{\mathcal{H}}_{l}$. when $t_{0}=0$ then we have $\ker(A)=\{0\}$ by the property 3 of lemma, and apply lemma
\ref{BasicLemma} to finish our proof of this case. If $t_{0} \ne 0$, we need the following estimation:
\begin{claim}
There are const $0<c<C<\infty$ such that for any $k, l>0$ we have
$$
c< \prod_{j=k}^{k+l} \frac{t_{0}-a_{j}}{t_{0}} <C.
$$
\end{claim}
In fact, as well known the infinite product $\prod_{j=1}^{\infty} \frac{t_{0}-a_{j}}{t_{0}}$ converges
if the series $\sum_{j=1}^{\infty} a_{j}$ converges(see, \cite{Stein}, p141). Now for the subspace $\widetilde{\mathcal{H}}_{l}$
and its ONB $\{e_{l}^{(k)}\}_{k=1}^{\infty}$, let $U_{l}$ be the unitary spread constructed in example
\ref{Example: Unitary Op is LC of Spreads}, and we denote the corresponding subsets by $A_{1}^{(l)}, B_{1}^{(l)}, A_{2}^{(l)}$
and $B_{2}^{(l)}$. By the theorem 2 in the paper \cite{Shields} and above claim, we know the operators $U_{l}\tilde{A}_{l}$
and $t_{0}U_{l}$ are similarity to each others. Hence the unitary spread $U_{l}$ satisfies
$\sigma_{p}(U_{l}\tilde{A}_{l})=\sigma_{p}(\tilde{A}_{l}U_{l}^{*})=\emptyset$. Let $U$ be the corresponding orthogonal direct sum
$U=\oplus_{l=1}^{\infty} U_{l}$, then we have $\sigma_{p}(UA)=\sigma_{p}(AU)=\emptyset$ and it is trivial to check that $U$
is also a unitary spread.

Case 2. Now we assume that $G_{1}$ is a finite subset of $\mathbb{N}$.
In virtue of lemma \ref{Lemma: spectrum is just a finite point spectrum} and again by the spectral theorem,
we can assume $G_{0}=\mathbb{N}$ and $G_{1}=\emptyset$. Just by the same reason, we can also assume $G_{0}$ is an infinite subset.
For convenience,
denote by $\alpha_{k}=\dim \widetilde{\mathcal{H}}_{k}$. If $\limsup_{k} \alpha_{k}=\infty$, then we can repeat our above discussion in
case 1 to finish the proof. So we just need to consider the situation $m=\limsup_{k} \alpha_{k}<\infty$.
Now we cut the subset $G_{0}$ into the pieces
$$
L_{n}=\{k \in \mathbb{N}; \alpha_{k}=n\}, 1 \le n \le m.
$$
Clearly there is at least one subset $L_{n}$ such that it is an infinite subset. We add all finite subset $L_{n}$ into a fixed
infinite subset, said, the set $L_{1}$. For the remaining infinite subsets
except $L_{1}$, we can repeat the discussion in case 1 to get an appropriate
unitary spread. For the infinite subset $L_{1}$, the same discussion also goes well if we apply the theorem 2 in the paper \cite{Shields}
again and note that adjusting finite nonzero weights into another nonzero ones does not change the similarity class of a weighted
bilateral shift operator. Hence we can also get a unitary spread which is a finite or infinite orthogonal direct sum of
unitary spreads dependent on the condition $\limsup_{k} \alpha_{k}<\infty$ or not, such that $\sigma_{p}(UA)=\sigma_{p}(AU)=\emptyset$.
\end{proof}

With a little more operator theory discussion, the proof of above lemma implies the following result. Since it deviate
our main aim in this paper, we omit the proof and just state it here.
\begin{theorem}\label{Theorem: in some case UA can be similarity as an infinite copies of unitary spread}
Assume that $A$ is a self-adjoint operator, satisfying the following properties:

1. $span \{\ker(t_{k}I-A); t_{k} \in \sigma_{p}(A)\}=\mathcal{H}$;

2. There is some point $t_{0} \ne 0$ such that it is an accumulation point of $\sigma(A)$.

3. $R(A)$ is dense in the Hilbert space $\mathcal{H}$.

Then there are unitary spreads $U, U_{\sigma}$ and an invertible operator $X \in L(\mathcal{H})$
such that $XUAX^{-1}=\oplus_{k=1}^{\infty}t_{0}U_{\sigma}$.
\end{theorem}

\begin{theorem}\label{Theorem: Discrete case of 2nd main theorem}
Assume that $A$ is a self-adjoint operator, satisfying the following properties:

1. $span \{\ker(t_{k}I-A); t_{k} \in \sigma_{p}(A)\}=\mathcal{H}$;

2. $R(A)$ is dense in the Hilbert space $\mathcal{H}$.

Then there is a unitary spread $U$ such that both $UA$ and $AU^{*}$ have an empty point spectrum.
\end{theorem}
\begin{proof}
If $\sigma_{p}(A)$ has no accumulation point then it is a finite set, and then we apply lemma \ref{Lemma: spectrum is just a finite point spectrum}. If not, above lemma \ref{Lemma: Point spectrum do have an Acc point} holds.
\end{proof}


\subsection{}Now we begin to consider the continuous case.

\begin{lemma}\label{Lemma: Continuous case with exactly one point spectrum}
Assume that $A$ is a self-adjoint operator satisfying the following properties:\\
1. $\sigma(A) \subseteq (m, M)$ in which $m<M$ are positive finite real numbers;\\
2. $\sigma_{p}(A)=\{t_{0}\}$ for some $t_{0} \in (m, M)$ and $\dim \ker(t_{0}I-A)=1$.\\
Then there is a unitary spread such that $\sigma_{p}(UA)=\sigma_{p}(AU^{*})=\emptyset$.
\end{lemma}
\begin{proof}
By the classical spectral theory of normal operator(cf, \cite{Conway}, pp297-299), we have following orthogonal
decomposition of $A$:
$$
\begin{array}{rl}
\left(
\begin{array}{cc}
t_{0} & 0 \\
 0    & A_{1}
\end{array}
\right) &
\begin{array}{l}
\ker(t_{0}I-A) \\
\ker(t_{0}I-A)^{\perp}
\end{array}
\end{array}.
$$
Then $A_{1}$ is a self-adjoint operator whose point spectrum must be void since
$\sigma_{p}(A)=\sigma_{p}(A_{1})\cup\{t_{0}\}$. Then $\sigma(A_{1})$ must be a closed set without isolated point
because each isolated point must be an eigenvalue of $A_{1}$ by the spectral theorem. Now we fixed a point
$\alpha \ne 0 \in \sigma(A_{1})$.
Then at least one of following assertions holds:\\
1. There is a sequence $\alpha_{n} \rightarrow \alpha$ such that we have $\alpha_{n+1}>\alpha_{n}$ for each $n \ge 1$.
Moreover, the range of spectral projection $E_{[\alpha_{n}, \alpha_{n+1}]}$ is an infinite subspace;\\
2. There is a sequence $\alpha_{n} \rightarrow \alpha$ such that we have $\alpha_{n+1}<\alpha_{n}$ for each $n \ge 1$.
Moreover, the range of spectral projection $E_{[\alpha_{n}, \alpha_{n+1}]}$ is an infinite subspace;\\
We assume that the first assertion is true. The case that the second assertion holds will be proved in the just same way.
Now let $\alpha_{1}=||A||$ for convenience.
By picking a subsequence if need, we also can assume that the sequence $\{\alpha_{n}\}_{n=1}^{\infty}$ satisfies the
following properties
$
\alpha(1-\frac{1}{2^{n}})<\alpha_{n}.
$
It is easy to check:

\begin{claim}\label{Claim1}
For each $\varepsilon >0$, there is a positive integer $N$ such that for any subset $\Delta$ containing $k$ elements of
$\mathbb{N}$ and satisfying $\Delta \cap \{1, 2, \cdots, N-1\}=\emptyset$ we have
$$
(1-\varepsilon)\alpha^{k}\le \prod_{n_{k} \in \Delta} \alpha_{n_{k}}\le(1+\varepsilon)\alpha^{k}.
$$
\end{claim}

Now we rearrange these interval as follows.
$$
\begin{array}{l}
I_{n}=[\alpha_{2n-1}, \alpha_{2n})\cup(2\alpha-\alpha_{2n}, 2\alpha-\alpha_{2n-1}] \hbox{ for } n \ge 1, \\
I_{n}=[\alpha_{-2(n+1)}, \alpha_{-2n-1})\cup(2\alpha-\alpha_{-2n-1}, 2\alpha-\alpha_{-2(n+1)}] \hbox{ for } n \le 0.
\end{array}
$$
Denote $E_{n}=E_{I_{n}}$($E_{r}=E_{I_{r}}$) the spectral projection of $A_{1}$ on the interval $I_{n}$($I_{r}$) and
by $\mathcal{H}_{n}=Ran (E_{n})$ for $n \in \mathbb{Z}$,
$\mathcal{H}_{0}^{'}=\mathcal{H}_{0} \cap \ker(t_{0}I-A)^{\perp}$ and $E_{0}^{'}$ be the orthogonal
projection onto the subspace $\mathcal{H}_{0}^{'}$.
Now we choose an ONB $\{e^{(n)}_{k}\}_{k=1}^{\infty}$ of $\mathcal{H}_{n}$ for each $n \in \mathbb{Z}, n\ne 0$.
For $\mathcal{H}_{0}$, we pick an ONB $\{\tilde{e}^{(0)}_{k}\}_{k=1}^{\infty}$ of the subspace
$\mathcal{H}_{0}^{'}$ and rearrange them and $e_{0}$ into an ONB of $\mathcal{H}_{0}$ as follows:
$$
e^{(0)}_{1}=e_{0}, e^{(0)}_{k}=\tilde{e}^{(0)}_{k-1} \hbox{ for }k \ge 2.
$$
It is trivial to check that the set $\varphi=\{e^{(n)}_{k}; n\in \mathbb{Z}, k \in \mathbb{N}\}$ is an ONB of the whole Hilbert space $\mathcal{H}$.
Now let $U$ be the unitary operator defined as
$$
Ue^{(n)}_{k}=e^{(n+1)}_{k}, \hbox{ for }n \in \mathbb{Z} \hbox{ and } k \in \mathbb{N}.
$$
By lemma \ref{Lemma: Bijection on N is a unitary Spread}, we know that $U$ is a unitary spread.

To finish the proof of lemma, now we prove that both $UA$ and $AU^{*}$ have no eigenvalues. The proof of these facts are similar,
so we only prove the first part and omit the other one to save space.
Since each $\mathcal{H}_{n}$ is a reducing subspace of $A$, we can write $A$ into the direct sum:
$$
A=\oplus_{n=-\infty}^{\infty} A_{n}
$$
in which $A_{n}=AE_{n}=E_{n}AE_{n}$ for $n \ne 0$ and $A_{0}=AE_{0}^{'}\oplus t_{0}I=E_{0}^{'}AE_{0}^{'}\oplus t_{0}I$.
We have the following estimation:
$$
\begin{array}{rl}
\alpha_{2n-1}||x|| \le ||A_{n}x|| \le  (2\alpha -\alpha_{2n-1})||x||, &\hbox{ for } n\ge 1\\
\alpha_{-2(n+1)}||x|| \le ||A_{n}x|| \le  (2\alpha -\alpha_{-2(n+1)})||x||, &\hbox{ for } n< 0.
\end{array}
$$
Moreover, by $\alpha<2\alpha-\alpha_{-1(n+1)}<(1+\frac{1}{2^{n}})\alpha$, we have
$$
(1+\frac{1}{2^{n}})^{-1}\alpha^{-1}||x||<(2\alpha -\alpha_{-2(n+1)})^{-1}||x|| \le ||A_{n}^{-1}x|| \le  \alpha_{-2(n+1)}^{-1}||x||
$$
for $n< 0$. For $A_{0}$, we have $m||x|| \le ||A_{0}x|| \le M||x||$.

For a vector $x \in \mathcal{H}$, now under the ONB $\varphi$ it have a $l^{2}-$ sequence coordinate in the form
$$
x^{(n)}=E_{n}x=\sum_{k=1}^{\infty}x^{(n)}_{k}e^{(n)}_{k} \in \mathcal{H}_{n},
x=\sum_{n \in \mathbb{Z}}x^{(n)},
$$
in which the series converges in the norm on $\mathcal{H}$
and $\{x^{(n)}_{k}\}_{k=1}^{\infty}$ is also a $l^{2}-$sequence. Here we emphasize that vectors $x^{(k)}$ and $x^{(j)}$
are orthogonal to each other for $k \ne j$. Let $y=Ax$, then by the fact that $\mathcal{H}_{n}$ is a reducing subspace
of the operator $A$ we can also write $y$ into the same form:
$$
y^{(n)}=A_{n}x^{(n)}=\sum_{k=1}^{\infty}y^{(n)}_{k}e^{(n)}_{k} \in \mathcal{H}_{n},
y=\sum_{n\in \mathbb{Z}} y^{(n)}.
$$

Now simply we have
$$
UAx^{(n)}=Uy^{(n)}=\sum_{k=1}^{\infty}y^{(n)}_{k}e^{(n+1)}_{k}.
$$
We can identify $\mathcal{H}_{n}$ with a fixed separable infinite dimensional
Hilbert space $\mathcal{H}^{*}$ as follows. Fix an ONB $\{e_{k}\}_{k=1}^{\infty}$, let $\widetilde{U}_{n}$ be the unitary operator defined as
$\widetilde{U}_{n}e^{(n)}_{k}=e_{k}$.
Now $\widetilde{U}_{n}x^{(n)}$ just the vector with the same $l^{2}-$coordinate
in $\mathcal{H}^{*}$. Moreover, each operator $A_{n}$ can be seen as the operator
$\widetilde{U}_{n}A_{n}\widetilde{U}_{n}^{*}$ in $L(\mathcal{H}^{*})$.
Hence $A$ is unitary equivalent to the operator $\oplus_{n \in \mathbb{Z}}\widetilde{U}_{n}A_{n}\widetilde{U}_{n}^{*}$ on the Hilbert space $\oplus_{-\infty}^{\infty}\widetilde{H}^{*}$ which is an orthogonal direct sum of countable copies of $\mathcal{H}^{*}$.
Denote by $\tilde{A}_{n}=\widetilde{U}_{n}A_{n}\widetilde{U}_{n}^{*}$ for convenience.

Now suppose for some $\lambda \ne 0$ we do have some vector $x$ such that
$(\lambda I-UA)x=0$, then we have
$$
\lambda \widetilde{U}_{n}x^{(n)}=\widetilde{U}_{n-1}A_{n-1}x^{(n-1)}
=\widetilde{U}_{n-1}A_{n-1}\widetilde{U}_{n-1}^{*}\widetilde{U}_{n-1}x^{(n-1)}
=\tilde{A}_{n-1}\widetilde{U}_{n-1}x^{(n-1)}.
$$
Therefore, following equations hold:
$$
\begin{array}{l}
\widetilde{U}_{n}x^{(n)}=\lambda^{-n} \tilde{A}_{n-1}\cdot \tilde{A}_{n-2} \cdots \tilde{A}_{0}\cdot \widetilde{U}_{0}x^{(0)},
\hbox{ for }n \ge 1; \\
\widetilde{U}_{n}x^{(n)}=\lambda^{|n|} \tilde{A}_{n}^{-1}\cdot \tilde{A}_{n+1}^{-1} \cdots \tilde{A}_{-1}^{-1}\cdot \widetilde{U}_{0}x^{(0)},
\hbox{ for }n \le -1.
\end{array}
$$
Immediately we have $x^{(0)} \ne 0$. Moreover, by the fact $||A_{k}^{-1}||=||\tilde{A}_{k}^{-1}||$ we have the following inequations:
$$
\begin{array}{l}
||x^{(n)}|| \ge m ||x^{(0)}|| \lambda^{-n} \prod_{k=1}^{n-1} \alpha_{2k-1}, \hbox{ for }n \ge 1; \\
||x^{(n)}|| \ge ||x^{(0)}|| \lambda^{|n|} \prod_{k=1}^{|n|} (1+\frac{1}{2^{n}})^{-1}\alpha^{-1} , \hbox{ for }n \le -1.
\end{array}
$$
Now for a given $\varepsilon >0$, let $N$ be the integer defined in the claim \ref{Claim1}, for $n \ge 1$ we have
$$
\begin{array}{rl}
||x^{(n)}|| &\ge m||x^{(0)}|| \lambda^{-n} (1-\varepsilon)\alpha^{n-N}\prod_{k=1}^{N-1} \alpha_{2k-1}  \\
&=(1-\varepsilon) m ||x^{(0)}|| \lambda^{-N}\cdot (\frac{\alpha}{\lambda})^{n-N}\prod_{k=1}^{N-1} \alpha_{2k-1}.
\end{array}
$$
And for $n \le 0$,
$$
\begin{array}{rl}
||x^{(n)}|| &\ge ||x^{(0)}|| \lambda^{|n|} \alpha^{n+N}\cdot \prod_{k=1}^{|n|} (1+\frac{1}{2^{n}})^{-1} \cdot \alpha^{-N} \\
&= ||x^{(0)}|| (\frac{\lambda}{\alpha})^{N}\cdot (\frac{\alpha}{\lambda})^{n+N}\prod_{k=1}^{|n|} (1+\frac{1}{2^{n}})^{-1}
\end{array}
$$

Now if $|\lambda| < |\alpha|$, then $(\frac{\alpha}{\lambda})^{n+N} \rightarrow \infty$ as $n \rightarrow \infty$.
If $|\lambda| > |\alpha|$, then $(\frac{\alpha}{\lambda})^{n+N} \rightarrow \infty$ as $n \rightarrow -\infty$
since the infinite product $\prod_{k=1}^{\infty} (1+\frac{1}{2^{n}})^{-1}$ converges to a nonzero number.
So we must have $|\lambda|= |\alpha|$.
But this implies
$$
||x^{(n)}|| \ge m||x^{(0)}||(1-\varepsilon)\lambda^{-N}\prod_{k=0}^{N-1} \alpha_{2k-1} \hbox{ for all }n >N,
$$
which is impossible since we have $||x||=\infty$ in such case.
\end{proof}

\begin{theorem}\label{theorem: Continuous case of 2nd Maintheorem}
Assume that $A$ is a self-adjoint operator satisfying:\\
1. $\sigma_{p}=\{\lambda_{1}, \lambda_{2}, \cdots, \lambda_{n}\}$ and $\dim \ker(\lambda_{k}I-A)<\infty$; and\\
2. $R(A)$ is dense in the Hilbert space $\mathcal{H}$.\\
Then there is a unitary spread $U$ such that $\sigma_{p}(UA)=\sigma_{p}(AU^{*})=\emptyset$.
\end{theorem}
\begin{proof}
Clearly we have $\lambda_{k} \ne 0$ for $k=1, 2, \cdots, n$. Let $m_{k}=\dim \ker(\lambda_{k}I-A)$.
We can assume that $m_{k}=1$ by adding $m_{k}-1$ copies
of $\lambda_{k}$ into $\sigma_{p}(A)$.
Denote by
$\mathcal{H}_{0}=span_{1\le k \le n} \{\ker(\lambda_{k}I-A)\}$ and
$\mathcal{H}_{r}=\mathcal{H}_{0}^{\perp}$. Moreover denote by
$A_{0}$ the diagonal operator $A_{0}=diag(\lambda_{1}, \cdots, \lambda_{n})$
and $A_{r}$ the restriction of $A$ on the reducing subspace $\mathcal{H}_{r}$.
By the spectrum theorem, we can write $A$ into the orthogonal
direct sum $A=A_{0} \oplus A_{r}$.

Case 1. Assume $0 \in \sigma(A)$. Then $\{0\}$ can not be an isolated point of $\sigma(A)$. And for ant $\delta>0$,
the projection $E_{\delta}$ on the interval $(-\delta, \delta)$ has an infinite dimensional range by the spectral theorem.
Let $\{\alpha_{k}\}_{k=1}^{\infty}$ be a sequence satisfies the following conditions:\\
1. $\alpha_{k}>\alpha_{k+1}$ and $\alpha_{k} \rightarrow 0$; \\
2. Let $I_{k}=[-\alpha_{k}, -\alpha_{k+1}) \cup (\alpha_{k+1}, \alpha_{k}]$, the spectral projection $E_{k}$
of $A_{r}$ on the subset $I_{k}$
has an infinite dimensional range; \\
3. $\cup_{k \ge 1}I_{k} \supset \sigma(A)-\{0\}$. \\
Now for $k \le n$, let $\tilde{A}_{r}^{(k)}$ be the orthogonal direct sum
$$
\begin{array}{rll}
 \tilde{A}_{r}^{(k)}=& \left(
        \begin{array}{cc}
         \lambda_{k} & 0   \\
            0    & E_{k}AE_{k}
        \end{array}
     \right) & \begin{array}{c}
                \ker(\lambda_{k}I-A) \\
                Ran(E_{k})
               \end{array}
\end{array}.
$$
Then $\tilde{A}_{r}^{(k)}$ is an operator on the subspace $\widetilde{H}_{k}=\ker(\lambda_{k}I-A) \oplus Ran(E_{k})$.
Moreover, for $k>n$ we define $\tilde{A}_{r}^{(k)}=E_{k}AE_{k}$. Now we see that each $\tilde{A}_{r}^{(k)}$ satisfies
the requirements of lemma \ref{Lemma: Continuous case with exactly one point spectrum}. So for each $k$, there is some
unitary spread $U_{k}$ on $\widetilde{H}_{k}$ such that we have $\sigma_{p}(U_{k}\tilde{A}_{r}^{(k)})=\sigma_{p}(\tilde{A}_{r}^{(k)}U_{k}^{*})=\emptyset$.
Moreover, again by the spectrum
theorem, we can write $A$ into the orthogonal direct sum $A=\oplus_{k=1}^{\infty} \tilde{A}_{r}^{(k)}$. Then the unitary
operator $U=\oplus_{k=1}^{\infty} U_{k}$ satisfies $\sigma_{p}(UA)=\sigma_{p}(AU^{*})=\emptyset$. It is easy to check that $U$
is a unitary spread by lemma \ref{Lemma: Bijection on N is a unitary Spread}.

Case 2. Assume $0 \notin \sigma(A)$. This situation is more easy to deal with. We just need to cut $\sigma(A_{r})$ into
exact $n$ suitable pieces and then repeat the above discussion.
\end{proof}

Now finally we can prove theorem \ref{Maintheorem3}.

\begin{proof}
Let $\mathcal{H}_{0}=span \{\ker(\lambda I-A); \lambda \in \sigma_{p}(A)\}$ and $\mathcal{H}_{1}=\mathcal{H}_{0}^{\perp}$.
By the spectrum theory of normal operator, we always can write $A$ into the form:
$$
\begin{array}{rl}
\left(
\begin{array}{cc}
A_{0} & 0\\
0 & A_{1}
\end{array}
\right)& \begin{array}{c}
          \mathcal{H}_{0} \\
          \mathcal{H}_{1}
         \end{array}
\end{array}
$$
in which $\mathcal{H}_{0}=span \{\ker(\lambda I-A_{0}); \lambda \in \sigma_{p}(A)\}$ and $\sigma_{p}(A_{1})=\emptyset$.
If $\dim \mathcal{H}_{0}=\infty$ we can apply theorem \ref{Theorem: Discrete case of 2nd main theorem};
In the case $\dim \mathcal{H}_{0}< \infty$ we apply theorem \ref{theorem: Continuous case of 2nd Maintheorem}.
\end{proof}


\begin{thebibliography}{000}

\bibitem{Arsove}Arsove, Maynard G. Similar bases and isomorphisms in Fr¨¦chet spaces. Math. Ann. 135, 1958, 283-293.

\bibitem{Beurling} Beurling, Arne On two problems concerning linear transformations in Hilbert space. Acta Math. 81, (1948).

\bibitem{Cao-1}
Y. Cao  G. Tian and B. Z. Hou, Schauder Bases and Operator Theory, preprint. Avaliable at http://arxiv.org/abs/1203.3603.

\bibitem{Cao-2}
Y. Cao  B. Z. Hou and G. Tian, On unitary operators in spread form(in Chinese),  (Chinese) J. Jilin Univ. Sci., accepted.

\bibitem{Conway} John B. Conway, A course in functional analysis, GTM96, Springer-Verlag, 1985.

\bibitem{Douglas} Douglas, Ronald G., Banach algebra techniques in operator theory. Pure and Applied Mathematics, Vol. 49. Academic Press,
New York-London, 1972.

\bibitem{Douglas2}  Cowen, M. J.; Douglas, R. G. Equivalence of connections. Adv. in Math. 56 (1985), no. 1, 39-91.

\bibitem{Douglas3} Cowen, M. J.; Douglas, R. G. Complex geometry and operator theory. Acta Math. 141 (1978), no. 3-4, 187-261.

\bibitem{Garling}Garling, D. J. H., Symmetric bases of locally convex spaces, Studia Math. 30, 1968, 163-181.

\bibitem{Halmos} Halmos, Paul Richard, A Hilbert space problem book. Second edition. Graduate Texts in Mathematics, 19.
Springer-Verlag, New York-Berlin, 1982.

\bibitem{J-M}Stephane Jaffard and Robert M. Young, A Representation Theorem for Schauder Bases in Hilbert Space, Proc. AMS.,
Vol. 126, No. 2 (Feb., 1998), pp. 553-560.

\bibitem{Ji-Tian-Cao}
Y. Q. Ji  G. Tian and Y. Cao, Strongly Irreducible Schauder Operators, preprint.

\bibitem{Jiang} Jiang, Chunlan; Ji, Kui, Similarity classification of holomorphic curves. Adv. Math. 215 (2007), no. 2, 446-468.

\bibitem{Jiang2} Jiang, Chunlan Similarity classification of Cowen-Douglas operators. Canad. J. Math. 56 (2004), no. 4, 742-775.

\bibitem{Jiang3} Jiang, Chunlan; Wang, Zongyao Structure of Hilbert space operators. World Scientific Publishing Co. Pte. Ltd., Hackensack, NJ, 2006.

\bibitem{Jiang4}  Jiang, Chunlan; Wang, Zongyao Strongly irreducible operators on Hilbert space. Pitman Research Notes in Mathematics Series, 389. Longman, Harlow, 1998.

\bibitem{Zhu} Zhu, Kehe, Operators in Cowen-Douglas classes. Illinois J. Math. 44 (2000), no. 4, 767-783.

\bibitem{Kwapien}  Kwapien, S.; Pelczynski, A. The main triangle projection in matrix spaces and its applications.
Studia Math. 34 1970 43-68.

\bibitem{B} Robert E. Megginson, An introuduction to Banach Space Theory, GTM183, Springe-Verlag, 1998.

\bibitem{Ole}A. M. Olevskii, On operators generating conditional bases in a Hilbert space, Translated
from Matematicheskie Zametki, Vol(12), No.1, pp. 73-84, July, 1972.

\bibitem{Y2}Young, Robert M. An introduction to nonharmonic Fourier series.
Pure and Applied Mathematics, 93. Academic Press, Inc. , New York-London, 1980.

\bibitem{Singer} I. Singer, Bases in Banach Space I, Springer-verlag, 1970.

\bibitem{Shields} Allen L. Shields, ``Weighted shift operators and analytic function theory'', in: Topics
in Operator Theory, Math. Surveys No. 13, 49-128, Amer. Math. Soc., Providence (1974).

\bibitem{Op} Joachim Weidmann, Linear Operators in Hilbert Space, GTM68, Springer-Verlag, 1980.

\bibitem{Wright} Wright, J. D. Maitland, All operators on a Hilbert space are bounded. Bull. Amer. Math. Soc. 79 (1973), 1247-1250.







\bibitem{Stein} Elias M. Stein and Rami Shakarchi, Complex analysis, Princeton Lectures in Analysis, Princeton University Press, 2003.

\bibitem{McCarthy} McCarthy, John E. Boundary values and Cowen-Douglas curvature. J. Funct. Anal. 137 (1996), no. 1, 1-18.


\bibitem{Gowers}Gowers, W. T.; Maurey, B. The unconditional basic sequence problem. J. Amer. Math. Soc. 6 (1993), no. 4, 851-874.

\bibitem{Gowers2}  Gowers, W. T.; Maurey, B. Banach spaces with small spaces of operators. Math. Ann. 307 (1997), no. 4, 543-568.

\bibitem{Cao} Cao Yang and el, On unitary operators spread $\mathbb{N}$, preprint.

\bibitem{Ivan-Hilbert-Hugh}Niven, Ivan; Zuckerman, Herbert S.; Montgomery, Hugh L. An introduction to the theory of numbers. Fifth edition. John Wiley and Sons, Inc., New York, 1991.
\end{thebibliography}
\end{document}